\documentclass[11pt]{amsart}

\usepackage{
amsfonts,
latexsym,
amssymb,
enumerate,
verbatim,
mathrsfs,
color,
}

\definecolor{bluuu}{RGB}{0,0,50}

\newcommand{\labbel}[1]{\label{#1} [[{\bf #1}]]}  
\newcommand{\bibbitem}[1]{\bibitem{#1} [[{\bf #1}]]}  
\renewcommand{\labbel}{\label} \renewcommand{\bibbitem}{\bibitem}  

\newcommand{\arxiv}[1]{{\color{bluuu}#1}}

\numberwithin{equation}{section}

\newtheorem{theorem}{Theorem}[section]
\newtheorem{lemma}[theorem]{Lemma}

\newtheorem{proposition}[theorem]{Proposition} 
 
\newtheorem{corollary}[theorem]{Corollary}

\newtheorem*{claim*}{Claim}

\newtheorem*{theorem*}{Theorem}
\newtheorem*{proposition*}{Proposition}
\newtheorem*{corollary*}{Corollary}
\newtheorem*{lemma*}{Lemma}
\newtheorem*{scholion*}{Scholion}

\theoremstyle{definition}
\newtheorem{definition}[theorem]{Definition}

\newtheorem{problem}[theorem]{Problem} 
\newtheorem{problems}[theorem]{Problems} 

\theoremstyle{remark}
\newtheorem{remark}[theorem]{Remark}

\newtheorem*{remark*}{Remark}
\newtheorem*{remarks*}{Remarks}
 
\newtheorem{example}[theorem]{Example}
\newtheorem{examples}[theorem]{Examples}

\newtheorem*{observation*}{Observation}

\newcommand{\ceq}{\mathrel{\stackrel{\scriptstyle  c}{=}}}

\newcommand{\eg}{\mathrel{\scriptstyle \stackrel{  g}{\equiv}}}

\newcommand{\eqi}{\mathrel{\scriptstyle \stackrel{ i}{\equiv}}}
\newcommand{\leqi}{\mathrel{\scriptscriptstyle \stackrel{  i}{\leq}}}

\newcommand{\eqa}{\mathrel{\scriptstyle \stackrel{\scriptstyle  a}{\equiv}}}
\newcommand{\leqa}{\mathrel{ \scriptscriptstyle  \stackrel{ a}{ \leq}}}

\begin{document}
 
\title[Ring structure on combinatorial games]
{A Ring structure on the Class of Combinatorial Games}

\author{Harry Altman}

\author{Paolo Lipparini} 
\urladdr{http://www.mat.uniroma2.it/~ lipparin}
\address{Dipartimento Giocoso di Matematica\\Viale della  Ricerca
 Scientifica\\Universit\`a di Roma ``Tor Vergata'' 
\\I-00133 ROME ITALY (currently retired)\\
ORCID 0000-0003-3747-6611}

\keywords{Combinatorial game; Conway product; iteratively second winner game;
commutative Ring; sets as games}

\subjclass[2020]{91A46; 13F70; 00A30}
\thanks{Work performed under the auspices of G.N.S.A.G.A. Work 
 supported by PRIN 2012 ``Logica, Modelli e Insiemi''.
Paolo Lipparini acknowledges the MIUR Department Project awarded to the
Department of Mathematics, University of Rome Tor Vergata, CUP
E83C18000100006.}

\begin{abstract}
J.\ Conway defined useful operations on the Class
of combinatorial games
and  also introduced a  notion of equivalence between games.
Conway showed that, under his equivalence, games form a Group.
However, Conway product  is not  well defined
on equivalence classes of arbitrary games
(though it is well defined for surreals).

We consider an equivalence relation
finer than Conway's and show that
under such a relation combinatorial games actually form a
Ring. We hint to other possible relations  on the Class of combinatorial games.
\end{abstract} 

\maketitle

\section{Introduction} \labbel{intron} 

Conway surreal numbers form a Field which contains
both the real numbers and the Ordinals,
and thus has a very rich structure. 
Combinatorial games,
earlier   studied in particular
cases by various authors,
have an even richer structure.
Conway defined a natural sum operation
on combinatorial games, and then also a product.
In order to get a Group structure, one
should work with respect to some equivalence relation,
since, say, $G + (-G)$ is a game more complex than $G$,
hence it is not isomorphic to $0$, unless $G=0$.   
Under a naturally defined equivalence relation,
also introduced by Conway, ``numbers'' do form
a Field, the mentioned Field of surreals.
However, for arbitrary games
which are not  ``numbers'', product is not well defined up to 
Conway equivalence.

User Gro-Tsen from mathoverflow.net \cite{GT} 
proposed a finer equivalence relation
(see Definition \ref{grotsen} below) and
asked whether it leads to an interesting theory.
We will provide partial affirmative answers
to Gro-Tsen's question in Section \ref{Gro-Tsen}, 
but here we mainly deal with a 
possibly distinct equivalence relation which 
makes the class of all Games a commutative Ring.
Recall that a game is Conway zero if the Second player
has a winning strategy. Here we study a finer relation
of being \emph{iteratively zero}, in the sense that the Second player
has a winning strategy under the following way of playing:
First and Second alternate in moving, as usual,
but, at each move made by her, the First player 
may choose to impersonate Left or Right, namely,
First, at her turn, is allowed each time to change her role.
On the other hand, Second is not allowed to choose his role:
each time First moves as Left, Second must reply as Right, and symmetrically.

Two games $G$ and $H$ 
are \emph{iteratively equivalent}, written as $\eqi$,
if $G- H$ is iteratively zero. Under this equivalence
relation, the Class of combinatorial games becomes a Ring $\mathbf{PG} / {\eqi}$
(Theorem \ref{ringth}), the main point being
that the product is well defined, a fact which follows
from the result that the product of an iteratively zero game
 with an arbitrary game 
 is still iteratively zero (Theorem \ref{thmi}).
Surreal numbers can be obtained as a Quotient of some
Substructure of $\mathbf{PG} / {\eqi}$ (Proposition  \ref{inumbsur});
moreover, the Universe of sets, thought of as
the Class of  Right-subposition-free games,
can be embedded into $\mathbf{PG} / {\eqi}$ (Proposition \ref{propsets}).
We also hint to other possible equivalence relations; in particular,
we show that, for any given set of operations on $\mathbf{PG} $, there is the coarsest
equivalence relation which respects each operation and at the same
time preserves outcomes (Proposition \ref{equiSex}).

In conclusion, the Class $\mathbf{PG} $
with the Conway operations, possibly with further operations
and possibly up to some comparatively fine equivalence relation,
seems to have an unexpected and  surprisingly rich structure 
which surely deserves further study.

We list a few problems,
the most relevant one seems to be whether
Gro-Tsen's relation is the same as $\eqi$.

\section{Preliminaries} \labbel{preln}

We assume that the reader is familiar with the basic notions of 
combinatorial game theory, e.~g., \cite{C,S}.
Here we just recall some basic definitions and fix
notation.

A \emph{combinatorial game} is (a position in)
 a possibly infinite
two-player  game with perfect information, 
no chance element, no draw
and no infinite run, as studied, for example, in \cite[Chapter VIII]{S}.
Games are possibly \emph{partizan}
and we always assume \emph{normal} (not mis\'ere) \emph{play}.  
See Conway \cite{C} and Siegel \cite{S} for full details.  
We generally follow the notation from \cite{S}; all
variations will be explicitly mentioned.
The class of all partisan combinatorial games is denoted 
by $\mathbf{PG}$.
As usual, the \emph{sum} of two games  is intended in the disjunctive
sense: moving on $G+H$ means  moving in exactly one component, either
$G$ or $H$. The two players are called Left and Right;
games are frequently denoted by a two-sided
set-theoretical-like notation, for example, $\{ \, 0, 1 \mid 2 \,\}$
is the game in which the Left player has two options $0$ and $1$,
while the Right player has the only option $2$. Recall that $0$
is the game $\{ \,  \mid  \,\}$ in which no player has any option,
$1$ is $\{ \, 0 \mid  \,\}$ and $2$ is $\{ \, 1 \mid  \,\}$.    

We use letters like $G$, $H$ for \emph{forms} (not values!) 
of games, that is, we consider games
up to
isomorphism;
 $\cong$ denotes isomorphism of games,
up to extensional equivalence\footnote{ Isomorphic 
games are called \emph{identical} in \cite{C},
but we will use a different terminology, as we are going
to explain soon.}.   
This means that, for example,  $ 2= \{ \, 1 \mid  \,\}$
and $1+1$ are isomorphic games, since they 
have isomorphic options. Formally, they are not
identical, since the option $1$ appears two times in 
$1+1$, while it appears only one time in $2$.    
 The tree associated \cite[p. 60]{S}   to $1+1$ has two Left edges  
exiting form the root, while only one edge exits from 
the root of $2$. 

In many texts, games are considered up to \emph{Conway equivalence},
which is commonly denoted by $=$.
Roughly, two games are Conway equivalent if
they give the same outcome, even when some fixed game
is added to both of them.
Here we will introduce
 more equivalence relations; moreover,
at the start, we are considering games just 
up to the much finer relation of
isomorphism, hence, in order to avoid 
notational confusion, we use $=$ for identity of games,
and we will denote Conway equivalence by $\ceq$. 

We say that some  relation $\equiv_1$
is \emph{finer} than   $\equiv_2$ if they have
the same domain and $a \equiv_1 b$
implies $a \equiv_2 b$, for all $a,b$ in the domain.
Under the same condition, we will say that 
$\equiv_2$
is \emph{coarser} than   $\equiv_1$.
Note that some authors use the above terms
with a reversed meaning.

If $G$ is a game, $-G$ is the game
in which all Left options are recursively exchanged with all
Right options.   
As usual, we write
$G - H$ for $G + (- H)$.
If $G$ and $H$ are games, their Conway product
$G H$ has Left options of the form
$G^LH + GH^L-G^LH^L$ and $G^RH + GH^R-G^RH^R$
and Right options of the form
$G^LH + GH^R-G^LH^R$ or $G^RH + GH^L-G^RH^L$,
where $G^L$ denotes a generic Left option of $G$,
$G^R$ denotes a generic Right option of $G$ and similarly for $H$.
Conway product will be denoted by juxtaposition
and, in some cases, by $\cdot$ for clarity. 

\begin{proposition} \labbel{basicn}
For all $G, H, K \in \mathbf{PG}$, the following 
statements and identities hold.
  \begin{enumerate}[({a}1)]    
\item
$\cong$ is an equivalence relation on  $\mathbf{PG}$
 and  $\cong$ is coarser than identity of games.
\item
If $G \cong H$ and $K \cong L$,
then $G+K \cong H + L$, $-G \cong -H$
and $G K \cong H L$.    
\item
$G+0 = G$ and $-0 = 0$, 
\item
$G+H = H+ G$,
\item
$(G+H) + K =  G + (H+K)$,
\item
$-(G+H) = (-G) + (-H)$,
\item
$-(-G) = G$,  
\item
$G0= 0$,
\item
$G1= G$,
\item
$GH= HG$,
\item
$(-G)H=-GH$.  
 \end{enumerate}
 \end{proposition}

 \begin{proof} 
Straightforward by ``$1$-line proofs'', compare \cite[p. 17]{C}.  
 See also  \cite[Theorems 3, 4 and 7 in Ch.\ 1]{C} or 
\cite[Propositions 1.5 (b) and 2.2 (a) - (c) in Ch.\ VIII]{S}. 
\end{proof}

\section{Iteratively zero games} \labbel{itz}

\begin{definition} \labbel{iter}
We say that a  game $G$ is \emph{iteratively zero},
 or \emph{iteratively second winner}
if the Second player has a winning strategy such that he
always chooses iteratively zero options.
In detail, $G$ is iteratively zero if, 
for every Left option  $G^L$ of $G$,
there is a Right option $G^{LR}$ of $G^L$
such that $G^{LR}$ is iteratively zero and,
symmetrically, 
for every Right option  $G^R$ of $G$,
there is a Left option $G^{RL}$ of $G^R$
such that $G^{RL}$ is iteratively zero.
For short, we might simply say that, for every First option  $G^F$ of $G$,
there is a Second option $G^{FS}$ of $G^F$
such that $G^{FS}$ is iteratively zero.

The definition is justified by transfinite induction on birthdays.
The game $0$ is iteratively zero, since it has no option,
and by the convention that a universal quantification
over the empty set is always true.
Given an arbitrary game $G$ of birthday $b(G)$, all the subpositions
of $G$ have birthday  strictly smaller than $b(G)$, so that 
we inductively know which subpositions are 
iteratively zero and which  subpositions are not. 

In an iteratively zero game the Second player has a winning strategy
such that he always moves to Conway zero games, but
we are asking a bit more: 
in any option chosen by Second,
the next player, whoever she is (namely, here we allow
possibly non alternating runs between Left and Right),
looses the game; again, in such a way
that the other player always can move to Conway zero games, actually,
in such a way that the recursive condition is satisfied.
 \end{definition}   

Of course, the above definition works up to
extensional equivalence, namely, if 
$G \cong H$, then $G$ is iteratively zero 
if and only if $H$ is.
As we just mentioned, every iteratively zero game
is second winner, that is, a Conway zero game.
The reverse implication does not necessarily hold, as
we are going to show.

\begin{examples} \labbel{exic}
(a) Let $2= \{ \, 1 \mid  \,\}$
and  $2^\circ = \{ \, 0, 1 \mid  \,\}$.
The game $ G = \{ \, -1 \mid  2^ \circ  \,\}$ 
is iteratively zero, since Left can only move
to $-1 = \{ \,  \mid 0 \,\}$, so Right moves to the zero game $0$.
If Right moves first, she moves to $2^ \circ $,
so Left can move to $0$ (though it is not his best possible
move!). 

On the other hand, 
$ H = \{ \, -1 \mid  2  \,\}$ 
is not iteratively zero, since, as above,
if Right moves first, she must move to $2$,
but then  Left  is forced to move to $1$.
Of course, Left wins the play, but he cannot accomplish the condition
of always moving to iteratively zero games, in particular, Conway zero. 
 
(b) As a possibly more significant example, 
$K= 1/2 + 1/2 -1$ is not iteratively zero, where
$1/2 = \{ \, 0 \mid 1 \,\}$. Indeed, Right can move
on $-1$ and Left cannot turn the game to the exact value $0$.
Of course, moving on $-1$ is not the best move for Right; in any case,
the condition of being iteratively zero is not satisfied.    

(c) Now consider the game 
$K^\bullet= \{ \,  0 \mid  K  \,  \mid \mid \,  0 \mid  0  \,\}$,
with $K$ as in (b).
Since $K \ceq 0$, we get   $K^\bullet \ceq 0$.
Moreover, each player, when moving second,
has a winning strategy such that he can always choose
options which are Conway zero games. 
Indeed, all the moves are forced, except when Left moves first,
with a forced Right reply to $K$. Then Left moves to
 $1/2 -1$ and Right wins moving to the Conway zero game 
$1-1$. 
 
However, $K^\bullet$ is not iteratively zero, 
since, as above,  when Left moves first,
Right must  reply to $K$, which is a zero game, but not
an iteratively zero game, by (b).
Hence the recursive condition in Definition \ref{iter} is not satisfied!

The point is that Left has started the game, with Right replying
choosing (necessarily) the option  $K$. If the play continues
with Left moving, Right can reply choosing the option
$1-1$, which is Conway zero, in fact, iteratively zero. 
However, in the definition of iteratively zero we require
that the above situation holds not only for (alternating)
plays, but also, say, for the \emph{run}   in which Right is expected
to move on $K$, namely, he makes two consecutive moves
(or, as in the introduction, we may assume that the players
are dubbed First and Second and they strictly alternate in playing, 
but First is allowed each time to change her role of Left or Right).
Then, as remarked in (b), Right can move
on $-1$ and Left cannot turn the game to the exact value $0$.
\end{examples}   

\begin{lemma} \labbel{lemi}
Let $G$ and $H$ be combinatorial games.
  \begin{enumerate}[({b}1)]
\item 
An impartial game is iteratively zero
if and only if it is Conway zero.
\item
If  $G$ is iteratively zero,
then  $-G$ is iteratively zero.  
\item
If both $G$ and $H$ are iteratively zero,
then $G+H$ is iteratively zero.  
\item
The game $H-H$ is iteratively zero. 
\item
If both $H$ and $G+H$ are iteratively zero,
then  $G$ is iteratively zero.   
  \end{enumerate} 
 \end{lemma}

 \begin{proof}
(b1) All subpositions of an impartial game are impartial;
moreover, an impartial game is either first winner or Conway zero.
So if some Player has a winning strategy, all the positions
she chooses
are impartial Conway zero positions, hence, by induction,
iteratively zero.

(b2)
The conclusion is straightforward: just exchange the
role of Left and Right.

(b3) is by induction
on the natural sum of the ordinals $b(G)$ and $b(H)$.
See \cite[p. 469]{S} for a quick introduction to ordinal natural sum.
For short, the \emph{natural (or Hessenberg) sum} of the ordinals $\alpha$ and
$\beta$ is obtained by expressing $\alpha$ and $\beta$ in Cantor
normal form and summing them as if they were polynomials in $ \omega$.
The natural sum is strictly monotone on both arguments; this fact
justifies the induction.

If the First player moves on, say, $G$, choosing the option
$G^F$, then, by assumption, Second has an option $G^{FS}$ 
which is an iteratively zero game. By induction,  
$G^{FS} + H$ is iteratively zero.

(b4) Again by induction, the standard mirror image strategy is iteratively winning.

(b5) By induction on the birthday of $G$.
Let First make the move $G^{F_1}$  on $G$;
this corresponds to a move $G^{F_1}+ H$ on $G+H$.
By the assumptions, Second can chose an option
in  $G^{F_1}+ H$ such that this option is iteratively zero.
If Second's option has the form 
 $G^{F_1S_2}+ H$, then  $G^{F_1S_2}$ is iteratively zero by the
inductive assumption. Thus the option
$G^{F_1S_2}$ is an iteratively zero Second option of $G^{F_1}$.

Otherwise, Second's option in 
$G^{F_1}+ H$ has the form $G^{F_1}+ H^{S_2}$ and 
$G^{F_1}+ H^{S_2}$ is iteratively zero.
Now we can consider Second as the first player on $H$.
Since $H$ is iteratively zero, First (now the second player on $H$)
can choose an iteratively zero position $H^{S_2F_3}$.
Since  $G^{F_1}+ H^{S_2}$ is iteratively zero, Second can move
on $G^{F_1}+ H^{S_2F_3}$ getting an iteratively zero position.
If Second's option has the form 
$G^{F_1S_4}+ H^{S_2F_3}$, then, again by the inductive assumption,
$G^{F_1S_4}$ is iteratively zero.
Otherwise, Second's option has the form
$G^{F_1}+ H^{S_2F_3S_4}$. Again, First has an iteratively zero option
$H^{S_2F_3S_4F_5}$ in $H^{S_2F_3S_4}$. If Second's
reply on $G^{F_1}+ H^{S_2F_3S_4F_5}$ is a move on the first component,
we get again an iteratively zero game $G^{F_1S_6}$ as above.
Otherwise, the game continues on the second component.
Going on this way, since First plays as second in the
iteratively zero game $H$, she eventually wins the play
on $H$, hence sooner or later Second is forced to move on the first component
and we can argue as above.

In any case, we have got 
an iteratively zero Second option  $G^{F_1S_n}$ of $G^{F_1}$.
\end{proof}

\begin{theorem} \labbel{thmi}
For all combinatorial games $G$ and $H$,
if $G$ is iteratively zero, then
$GH$ is iteratively zero, in particular,  Conway zero.
 \end{theorem}

 \begin{proof}
By induction on the birthday of $G$ 
and a subinduction on the birthday of $H$.
So, given $G$ and $H$, let us assume that the implication 
holds for every pair $G^\diamondsuit$,  $H^\diamondsuit $, with
 $b (G^\diamondsuit) < b(G)$ and arbitrary $H^\diamondsuit$.
Moreover, assume that the implication holds for the
given $G$ and for arbitrary $H^\heartsuit$ with
  $b (H^\heartsuit) < b(H)$.

 A typical Right option of $GH$
is 

($\bullet$)  $ G^LH +  GH^R  - G^LH^R$.

Then Left can choose the option
 $G^{LR}H + G^LH^R - G^{LR}H^R$ 
of $G^LH  $, where  the option
$G^{LR}$ of $G^{L}$ is such  that
$G^{LR}$ is iteratively zero.
The resulting option in the sum ($\bullet$) is
$  G^{LR}H + G^LH^R - G^{LR}H^R + GH^R- G^LH^R$.
By the inductive assumption, both $G^{LR}H $ and
$G^{LR}H^R $ are  iteratively zero.
Again by the inductive assumption,
$GH^R$ is iteratively zero. The remaining summands are
$G^LH^R$ and $-G^LH^R$, whose sum is iteratively zero,
by Lemma \ref{lemi}(b4).
In conclusion, the option chosen by Left 
in ($\bullet$) is iteratively zero,
by Lemma \ref{lemi}(b3).
  
All the remaining cases 
(considering all the other possible First options of $GH$) are treated in a similar way.
\end{proof}

\begin{remark} \labbel{notC}
The analogue of Theorem \ref{thmi}
fails for Conway equivalence, actually, there is 
 a game $K$ such that $K \ceq 0$ but not  $K^2 \ceq 0$.
In particular, this provides a strong (counter)example 
showing that, for general combinatorial games,
  product is not well-defined with respect to Conway 
equivalence. 

Indeed, set  $G = \{ \, -1,0 \mid 0,1 \, \}$,  $H =  * =  \{ \, 0 \mid  0 \, \} $
and $K=G+H$. The game $K$ 
is easily shown to be second-winner, that is,
$K \ceq 0$ 
(in fact, 
$G$ and $H$ are Conway equivalent.) 
Let us compute 
$* \cdot * =  *$,
$HG = GH  =   G \cdot * = \{ \,  0, * \mid 0, * \,\} $, hence
\begin{equation*}  
K^2 = (G+H)(G+H) \ceq 
 G^2 + GH + HG + H^2 \ceq  G^2 +  \{ \, 0 \mid  0 \, \},
\end{equation*}     
where we have used the fact that Conway product
is distributive, up to Conway equivalence \cite[Theorem 7]{C}, \cite[VIII, 
Proposition 2.2(e)]{S} (compare also Proposition \ref{lemip} below).
 Considering the Left options $-1$, $0$ 
 in the definition of the Conway product $G \cdot G$ we get that 
$ G \cdot (-1) + 0  \cdot G - 0 \cdot  (-1) = -G = G \ceq * $
is a Left  option of $G^2$; similarly
$G \ceq *$  is a Right  option of $G^2$.
There is no need to fully compute $G^2$ to see that 
$G^2 +  \{ \, 0 \mid  0 \, \}$ is a first-winner game (hence not $ \ceq 0$),
since the first player can always move in  $G^2$
choosing the option equivalent to $*$. The second player is left to play with
a game equivalent to
$*+*$, so he loses. 
 \end{remark}

\section{An equivalence relation making $\mathbf{PG}$ a Ring} \labbel{ringsec}

\begin{definition} \labbel{ggd}
For combinatorial games $G$ and $H$ 
we say that $G$ and $H$ are \emph{iteratively equivalent},
in symbols,  $G \eqi H$ if the game
$G - H$ is iteratively zero.

In particular, $G \eqi 0$
if and only if $G$ is iteratively zero, 
by Proposition \ref{basicn}(a3). 
Moreover,
if  $G \eqi H$,
then  $G \ceq H$, that is, $\eqi$ is finer than $\ceq$.  
 \end{definition}   

\begin{proposition} \labbel{eqeqi}
  \begin{enumerate}[(i)]    
\item   
The relation $\eqi$ is an equivalence relation. 
\item
If $H$ is iteratively zero, then $G + H \eqi G$,
for every game $G$. 
 \item
The relation $\eqi$ respects $+$, namely, 
if   $G \eqi H$ and  $K \eqi J$, then 
$G +K  \eqi H+J$.
 \end{enumerate}
\end{proposition}

  \begin{proof}
(i) Reflexivity of $\eqi$ 
follows from Lemma \ref{lemi}(b4)
and symmetry follows from (b2), since 
$H - G = -(G-H)$,
by Proposition \ref{basicn}(a4)(a6)(a7).

If $G \eqi H$ and $H \eqi K$, 
then $G - H $ and 
 $H - K $
are iteratively zero, 
hence  $G-H + H - K $ is iteratively zero,
by Lemma \ref{lemi}(b3)
(of course, we are implicitly using 
 Proposition \ref{basicn}(a5). From now on
we will not always explicitly mention when we use 
elementary facts from Proposition \ref{basicn}).
Thus  
$(G- K) + (H -H)$ is iteratively zero.
By Lemma \ref{lemi}(b4)(b5),
$G-K$ is iteratively zero, that is, 
$G \eqi K$.

(ii) follows straightforwardly
from Lemma \ref{lemi}(b3)(b4). 

(iii) If $G \eqi H$ and $K  \eqi J$ then 
both $G-H$ and $K-J$ are  iteratively zero,
hence
$G-H+K-J$ is iteratively zero, 
by Lemma \ref{lemi}(b3).
That is,  $G+ K -(H+J)$ is iteratively zero
and this means $G + K \eqi H+J$.
 \end{proof} 

For our next proposition, we'll need a definition:

\begin{definition}
\labbel{genetic-def}
We say an equivalence relation $\equiv$ on $\mathbf{PG}$ is
 \emph{option-regular}
 if, whenever we have games $G$ and $H$, such that
\begin{itemize}
\item for each left option $G^L$ of $G$, there is a left option $H^L$ of $H$
such that $G^L\equiv H^L$, and
\item for each left option $H^L$ of $H$, there is a left option $G^L$ of $G$
such that $H^L\equiv G^L$, and
\item for each right option $G^R$ of $G$, there is a right option $H^R$ of $H$
such that $G^R\equiv H^R$, and
\item for each right option $H^R$ of $H$, there is a right option $G^R$ of $G$
such that $H^R\equiv G^R$;
\end{itemize}
then $G\equiv H$.
\end{definition}

\begin{remark}
Note that the conditions above are equivalent to saying that, if we look at the
set of left options of $G$ modulo $\equiv$ and the set of left options of $H$
modulo $\equiv$, then there is a bijection between these with each left option
of $G$ matched to an equivalent left option of $H$; and the same thing for the
right options.
\end{remark}

In some previous works we have used the word ``genetic''
in place of ``option-regular'', but the present terminology
seems much clearer.    
The word ``genetic'', when applied to concepts in combinatorial game theory,
normally refers to a style of definition, rather than to a formal notion by
which one can say that a particular relation or operation is or is not genetic.
An equivalence relation with a genetic definition certainly ought to be
option-regular; but an equivalence relation can potentially be option-regular
without being given by a genetic definition.

Examples of well-known option-regular equivalence relations include isomorphism
$\cong$, Conway equivalence $\ceq$, and two games having the same comparison to
zero (positive, negative, zero, or fuzzy).

\begin{proposition} \labbel{gen}
The relation $\eqi$ is option-regular.
 \end{proposition} 

\begin{proof}
Let $G$ and $H$ have $\eqi$-equivalent options;
we want to show that $G \eqi H$, namely,
that $G-H$ is iteratively zero.
So let First make a move 
$G^F-H$ on $G-H$.
By assumption, there is a First option
$H^F$ on $H$ such that  
$G^F \eqi H^F$, namely,
 $G^F-(H^F)$ is iteratively zero. 
But $-(H^F)$ is a Second option\footnote{Here we are
 using the assumption in the definition of 
an iteratively zero game that if, at some point, First plays
as Left, then, in the next move Second plays as Right, and conversely.} of $-H$,
hence we get the conclusion.
The case when First makes a move 
$G+(-H)^F$ on $G-H$
is symmetrical.
 \end{proof}    

Another subtle technical result is necessary,
before we can prove our main theorem. 
The main nuisance is that the games 
$(G+H)K  $ and $  GK + HK$ 
are generally not isomorphic, since, say,
\begin{equation}\labbel{fop}     
(G^F+H)K + (G+H)K^F - (G^F+H)K^F
  \end{equation}
is a First option of the former game,  while 
\begin{equation}\labbel{fopp} 
      G^FK + GK^F - G^FK^F +  HK
  \end{equation}
is the ``corresponding'' option of the latter game. 

However, as well-known, $(G+H)K  $ and $  GK + HK$ 
are Conway equivalent. Since here we deal with
a relation finer than $\ceq$, we need to show that,
more generally, $(G+H)K  $ and $  GK + HK$ 
are $\eqi$-equivalent.  
This will be used in  order to prove that 
$\eqi$ respects product.

\begin{lemma} \labbel{lemip}
For all games $G, H, K$,
the following holds. 
\begin{equation}\labbel{quass}
(G+H)K \eqi GK + HK
   \end{equation}    
 \end{lemma} 

\begin{proof} 
Let us assume \eqref{quass} for games of smaller birthdays,
namely, work by induction on 
the natural sum of $b(G) $, $ b(H) $ and $ b(K)$. 
By the inductive hypothesis
and Proposition \ref{eqeqi}(iii),
 the expression in \eqref{fop} turns out to be
$\eqi$ equivalent to  
$G^FK+HK + GK^F+HK^F - G^FK^F -HK^F$,
which differs from \eqref{fopp} by 
$HK^F  - HK^F$.
By Lemma \ref{lemi}(b4) and  Proposition \ref{eqeqi} 
the expressions in \eqref{fop} and \eqref{fopp}
are $\eqi$-equivalent.  
 The other cases are proved in a similar way.
This shows  that the games in \eqref{quass}
have $\eqi$-equivalent options, hence they are
$\eqi$-equivalent, by Proposition \ref{gen}.
\end{proof}     

In particular, for every game $G$, we get
$2G \eqi G+G$,  $3G \eqi G+G+G$, etc., 
of course, provided we define $3 =\{ \, 2 \mid  \,\}$,
in general, $n+1 = \{ \, n \mid  \,\}$.  

So far, \eqref{quass} does not amount to show
that  we have the distributive property in some Ring,
since we have not yet showed that $\cdot$ is well-defined
on the equivalence classes. In fact, we will use 
\eqref{quass} in the proof.

\begin{theorem} \labbel{ringth}
The relation $\eqi$ is an option-regular equivalence relation
on the class $\mathbf{PG}$ of combinatorial games;
moreover, $\eqi$ respects both $+$ and $\cdot$ .
The quotient  $\mathbf{PG} / {\eqi}$ is a commutative
Ring with unity. 
 \end{theorem}

 \begin{proof}
By Propositions \ref{eqeqi}, \ref{gen},   
$\eqi $ is an option-regular equivalence relation which respects $+$.

Let $G \eqi H$ and $K$ be a game.
 Then $G-H$ is iteratively zero,
hence $(G-H)K$ is iteratively zero, 
by Theorem \ref{thmi}.
By \eqref{quass},
$GK-HK$  is $\eqi$-equivalent to   
$(G-H)K$, hence iteratively zero, 
by transitivity of  $\eqi$.
Thus $GK \eqi HK$.
This means that $\eqi$ respects $\cdot$,
using commutativity of $\cdot$. 

We have showed that $\eqi$ respects 
both $+$ and $\cdot$,
hence the operations are well-defined on the
equivalence classes.

Commutativity of the operations,
associativity of addition  and the existence of  neutral
elements follow from Proposition \ref{basicn}(a3)-(a5), (a9), (a10). 
By Lemma \ref{lemi}(b4) additive inverses exist.
Distributivity of multiplication with respect to addition
is Lemma \ref{lemip}.  
The classical argument showing that 
$G(HK)$ is Conway equivalent to
$(GH)K$ \cite[p. 19]{C}, \cite[VIII, Proposition 2.2(d)]{S}
does use an equivalence like \eqref{quass}, 
 though this aspect 
 goes generally unmentioned.
In any case, we did prove \eqref{quass},
hence the argument, together with transitivity of $\eqi$,
  provides associativity
modulo $\eqi$. 
 \end{proof}

\section{Remarks and problems about $\mathbf{PG} / {\eqi}$} \labbel{rmki}

\begin{remark} \labbel{notint} 
(a) The Ring
$\mathbf{PG} / {\eqi}$ is not  an Integral Domain.

For example, $2 * = (1+1)* \eqi 1*+1* = *+* \eqi 0$,
by distributivity and since $* = -*$, hence
 $*+* \eqi 0 $; recall that $ * = \{ \, 0 \mid 0 \,\}$.

(b)
More generally, if $G$ is a game such that 
there is some game $H$ such that $H+H \eqi G$,
then $G * \eqi 0$; actually,    $G Y \eqi 0$,
for every game $Y$ such that $Y+Y \eqi 0$.
Indeed,  under the above assumptions
$0 \eqi H 0 \eqi H(Y+Y) \eqi HY+HY \eqi (H+H)Y \eqi GY  $.

In particular, if $1/2$ is  $\{ \, 0 \mid 1 \,\}$,
then it is not the case that 
$1/2 + 1/2 \eqi 1$ (compare also
Example \ref{exic}(b)). Actually, $1$ is not divisible by $2$:
there is no game $H$ such that $H+H \eqi 1$.    
\end{remark}

The above remark 
could suggest the idea that 
$\mathbf{PG} / {\eqi}$, though a ring, is not actually well-behaved.
Be that as it may, we are going to show that 
  $\mathbf{PG} / {\eqi}$ is richer than the Field
of surreal numbers. Recall that $0 < G$ if Left has a winning strategy on $G$ 
whoever starts playing, and that
$H < G$ if $0 < G-H$.
Recall that a  \emph{number} 
is a game $G$ such that $H^L < H^R$,
for every (possibly improper) subposition $H$ of $G$
and all options  $H^L $ and $  H^R$ of $H$. 
A \emph{surreal number} 
is the $\ceq$-class of some number.
The Class of surreal numbers is a Field
containing both the ordinals and the real numbers
(it is fundamental to check that product is well-defined up to
Conway equivalence, for numbers). 

Let $\mathbf{Numb}/ {\eqi}$
be the subclass of $\mathbf{PG} / {\eqi}$ consisting of those
equivalence classes of the form $G/ {\eqi}$, for some number $G$.
Note that if $G$ is a number and $G \eqi H$, then $H$ 
is not necessarily a number, for example, $\{ \, * \mid * \,\} \eqi 0$,
but $\{ \, * \mid * \,\}$ is not a number.

\begin{proposition} \labbel{inumbsur}
$\mathbf{Numb}/ {\eqi}$ is a SubRing of
$\mathbf{PG} / {\eqi}$. The Class $\mathbf  C$ of Conway zero 
games in $\mathbf{Numb}/ {\eqi}$ is an Ideal,
and the Quotient by this Ideal is isomorphic to 
the Class of surreal numbers.
 \end{proposition}  

\begin{proof}
The first statement follows
from the result that   numbers are closed with respect  to sums and products,
and the opposite of a number is still a number.
Strictly formally, $\mathbf  C$ is a collection of
equivalence classes, but this is inessential, since 
$\eqi$ is finer than $\ceq$.  
$\mathbf  C$ is an Ideal, since the product of a Conway zero \emph{number}
with a number is still Conway zero  (working modulo $\eqi$ 
causes no trouble, since $\eqi$ respects $\cdot$, by Theorem \ref{ringth}).
We get surreal numbers as the quotient since two games $G$ and $H$ 
are Conway equivalent if and only if $G-H$ is Conway zero. 
\end{proof}

\begin{problems} \labbel{probsi} 
  \begin{enumerate}[(a)]    
\item
Study the structure of $\mathbf{PG} / {\eqi}$.
Is $\mathbf{Numb}/ {\eqi}$  an
Integral Domain?
By Proposition \ref{inumbsur}, $\mathbf  C$ is
a maximal Ideal of $\mathbf{Numb}/ {\eqi}$, since the Quotient is
a Field.  Study other  Ideals of $\mathbf{Numb}/ {\eqi}$.
Are there subrings $\mathbf S$  of $\mathbf{PG} / {\eqi}$
which are strictly larger than $\mathbf{Numb}/ {\eqi}$
and such that classes of Conway zero games in $\mathbf S$ 
still form an Ideal?

In another direction, 
working in $\mathbf{PG} / {\eqi}$ instead, 
the classes of games of the form
$G_1Z_1+ \dots + \dots G_mZ_m$, with $Z_1, \dots, Z_m$
Conway zero, constitute a (possibly quite large) Ideal.
Study the Quotient.
\item 
Is it true that a game $G$ is iteratively zero
if and only if 
$GH$ is Conway zero for every combinatorial game $H$?
(The ``only-if'' condition is true by Theorem \ref{thmi}.)
\item
Characterize   $\eqi$  as the finest (or coarsest)
equivalence relation $\equiv$ on $\mathbf{PG}$ satisfying a certain set
of conditions. Eligible properties are being option-regular,
satisfying $G \equiv H$ if and only if $G-H \equiv 0$,
or satisfying a set of conditions taken from 
Lemma \ref{lemi}, Theorem \ref{thmi}  
(for games  $\equiv 0$),
 Lemma \ref{lemip} and Theorem \ref{ringth}. 
\item
More specifically, is $\eqi$ the finest equivalence relation $\equiv$
which (under Conway's operations) makes  $\mathbf{PG} / {\equiv}$  a commutative
Ring? Is $\eqi$ the coarsest (the only) nontrivial relation $\equiv$ which 
makes  $\mathbf{PG} / {\equiv}$  a commutative
Ring? Of course, the answer to the second question 
depends on what we mean by ``non-trivial'',
for example, we might ask that $\equiv$ never identifies two distinct
ordinals.  In this respect, compare also Definition \ref{setofop} below.
\item
By induction, the product of an arbitrary game with an impartial
game is impartial. It follows that
the class of ($\eqi$-classes of games 
$\eqi$-equivalent to some) impartial games is an Ideal
of $\mathbf{PG} / {\eqi}$.
Study the corresponding Quotient.
\item
Describe the Ring of Endomorphisms $\mathbf R$ 
of the group $\mathbf{PG} / {\eqi}$
with game addition, possibly restricted
to appropriate ``good'' classes  
of endomorphisms, both 
to avoid set theoretical issues,
and  to get a more manageable Structure.
By Theorem \ref{ringth},  multiplication with a fixed
 game is an endomorphism of $\mathbf{PG} / {\eqi}$,
hence  $\mathbf{PG} / {\eqi}$
embeds in  $\mathbf R$.

We do not know whether the analogous problem
has been considered for just surreals. 
Automorphism groups of surreals with
respect to various structure have been studied in \cite{KKS}.
\item
Is there a natural definition of some preorder
$\leqi$ on   $\mathbf{PG}$  such  that 
 $\leqi$ induces $\eqi$, namely,
\begin{equation}\labbel{fif}     
\text{$G \eqi H$ if and only if 
both $G \leqi H$ and $H \leqi G$?  } 
  \end{equation}

In case the answer is affirmative, does $\leqi$ induce
the structure of a partially ordered Ring on 
$\mathbf{PG} / {\eqi}$?

\arxiv{Concerning the above question, we can set 
$ G \leqi H$ if in $H-G$ Left can always reply to any first
Right  move by choosing an iteratively zero option.
With this definition, \eqref{fif} actually holds, but 
 $\leqi$ does not seem to be transitive. Does the
transitive closure of $\leqi$ works?}
 \item
 Can we realize  some extension of $p$-adic
numbers, for some $p$, in a way similar to Proposition  \ref{inumbsur}?
 (Compare Question 3 in \cite{BEK}.)

\item Can we realize \emph{surcomplex numbers}
(ie, the Class of surreal complex numbers No[i] \cite[p. 42]{C})?

Of course, we can consider formal sums of  type
$G+Hi$, with the usual distributive rules and $i$ a new
formal object satisfying $i^2 = -1$.   Since we have showed
that $\mathbf{PG} / {\eqi}$ is a Ring, we can ``complexify''
$\mathbf{PG} / {\eqi}$ in the above way.
 It is not clear how relevant the construction can be,
for games which are not numbers. 

But the main point is that formal sums of  the kind
$G+Hi$ have not (yet) given a purely game theoretical meaning,
even when $G$ and $H$ are restricted to surreals.

\item Similarly, can we construct game theoretically 
surreal quaternions, octonions, more generally, Clifford algebras?
Of course, to this end, we should define another
product operation, since Conway product is commutative.
 \end{enumerate} 
\end{problems} 

\begin{remark} \labbel{phy}    
The above problems (i), (j) might be relevant to foundations of physics.
On one hand, surreal numbers have been sometimes proposed as a 
possible tool for dealing 
with physical divergences, e.~g., \cite[p. 162]{T}, \cite{CE}. 

On the other hand, what is relevant here is that
 there is a diffuse conviction  among physicists that
space (more precisely, space-time) has not a fundamental
physical nature, but should ``emerge'' from some
yet unknown more basic structure, sometimes dubbed \emph{pregeometry},
e.~g. \cite[\S 44.4]{MTW}.    
Surreal numbers of birthday $\leq \omega $
provide a very clean \emph{combinatorial} way 
of introducing the real line and, with different restrictions
on birthdays, many generalizations. Thus it is fair to say
that (some subset of) the surreal numbers might actually
constitute the ``pregeometry'' of the real line.
 
However, as we mentioned, higher dimensional generalizations
of the surreal numbers are obtained by rather artificial methods,
in any case, methods which are not actually game-theoretical.
If there is a natural generalization producing 
complex and quaternionic numbers, this would possibly be
a very interesting approach to the problem
of space emergence!
\end{remark}

\section{Further equivalence relations} \labbel{Gro-Tsen}

To the best of our knowledge, the relation
introduced in the following definition 
 first appeared in a question by User Gro-Tsen on 
mathoverflow.net \cite{GT}.

\begin{definition} \labbel{grotsen}
\cite{GT}
For all $G, H \in \mathbf{PG}$,
let $G \eg H$ if  
$G K \ceq H K$, for every 
  $K \in \mathbf{PG}$.
 Since $\ceq$ is an equivalence relation, also 
 $\eg$ is an equivalence relation. 

Let us say that a game $G$ is \emph{Gro-Tsen zero}
if $G \eg 0$, namely, if 
$G K \ceq 0$, for every game $K$.  
 \end{definition}  

\begin{corollary} \labbel{finer}
(i) The relation $\eqi$ is finer then $\eg$.

(ii) The relation $\eg$ is finer then $\ceq$.

(iii) The relations $\ceq$, $\eqi$ and $\eg$ all 
coincide on impartial games. 

(iv) The $\eqi$-classes of  Gro-Tsen zero games form an
ideal of the ring $\mathbf{PG} / {\eqi}$.

(v) In particular, $\mathbf{PG} / {\eg}$
is a Ring with the operations induced by $+$ and $\cdot$.  
 \end{corollary} 

 \begin{proof} 
(i) Suppose that  $G \eqi H$. Then
$GK \eqi HK$, for every game $K$, by Theorem \ref{ringth}. 
In particular,  $GK \ceq HK$, for every game $K$.
Thus $G \eg H$.

(ii) Just take $K=1$
in the definition of $\eg$.

(iii) is immediate from (i), (ii) and Lemma \ref{lemi}(b1).

(iv)  If both $G$ and $H$ are  Gro-Tsen zero, then $G +H$
is Gro-Tsen zero, since 
$(G+H)K  \ceq   GK + HK$, for all games $G$,  $H$, $K$.

If  $G$ is Gro-Tsen zero, then $GH$
is Gro-Tsen zero,  for every game $H$, since 
$(GH)K  \ceq   G(HK)$, for all games $G$,  $H$, $K$.

(v) is immediate from (i) and (iv).
\end{proof} 

\arxiv{
\begin{example} \labbel{3}   
The game $\{-3|\}$  is such that 
$\{-3|\} \cdot H$ is Conway zero, whenever $H$ 
 is either a number, or an impartial game.
Indeed, in the former case, 
$\{-3|\} \cdot H$ is Conway zero since 
$\{-3|\} $ is Conway zero and product is well-defined
for numbers, up to Conway equivalence.
On the other hand, if $H$ is impartial,
we check that $\{-3|\} \cdot H$ is Conway zero 
by induction on the birthday of $H$.
A typical option of $\{-3|\} \cdot H$ is
\begin{equation}\labbel{pirp}      
\{-3|\} \cdot H^L - 3 \cdot H + 3 \cdot H^L,
  \end{equation}
which is Conway equivalent to  
$\{-3|\} \cdot H^L - 3 \cdot (H -  H^L)$.
The first summand $\{-3|\} \cdot H^L$
is Conway zero, that is, Second winner, by
the inductive assumption.
The second summand is $(H -  H^L) + (H -  H^L) +(H -  H^L)$.
Since $H -  H^L$  is impartial, it is First winner, hence so is
the sum of three copies of $H -  H^L$. In conclusion,
\eqref{pirp} is First winner, so  that
$\{-3|\} \cdot H$ is Second winner (as we mentioned
in a comment in Problem \ref{probsi}(e), the product of an arbitrary game with an impartial game is impartial, hence $\{-3|\} \cdot H$ is impartial, so either
First or Second winner).

However, $\{-3|\}$ is not Gro-Tsen 0: if $H = \{0|0,*\}$, then a Left option in 
$\{-3|\} \cdot  H$ is $(-3) \cdot H$, which is Left winner, 
so $\{-3|\} \cdot H$ is not Second winner.

In the above example we can equivalently consider $\{-3|3\}$, or $\{|3\}$.

This example shows that any argument 
providing a characterization of Gro-Tsen relation 
should be rather delicate.
 \end{example}
}

Here's one more way that $\eg$ resembles $\eqi$:
\begin{proposition}
The equivalence relation $\eg$ is option-regular.
\end{proposition}

\begin{proof}
Suppose $G$ and $H$ are games such that for each left option $G^L$ of $G$ there
is a left option $H^L$ of $H$ such that $G^L\eg H^L$, and that the other
conditions of Definition~\ref{genetic-def} are similarly satisfied, and suppose
that $K$ is any game.  We wish to show that $GK\ceq HK$; in doing so we may
induct on the structure of $K$, and assume that the statement is true for all
options of $K$.

A left option of $GK$ has the form $G^L K + G K^L - G^L K^L$, where $G^L$ is a
left option of $G$ and $K^L$ is a left option of $K$.  Now, there is some left
option $H^L$ of $H$ with $H^L \eg G^L$.  Then $G^L K \ceq H^L K$ and $G^L K^L
\ceq H^L K^L$.  Also, $G K^L \ceq H K^L$ by the inductive hypothesis.  So, since
addition and negation respect $\ceq$, we conclude that
\[ G^L K + G K^L - G^L K^L \ceq H^L K + H K^L - H^L K^L; \]
the right-hand side is a left option of $HK$.  We may then apply the same logic
to the left options of $HK$, the right options of $GK$, and the left options of
$GK$.  Since $\ceq$ is well-known to be option-regular, we conclude that $GK\ceq HK$.
And since $GK\ceq HK$ for all $K$, we conclude that $G\eg K$.
\end{proof}

\begin{problems} \labbel{probs} 
  \begin{enumerate}[(a)]    
\item
Does $\eqi$ coincide with $\eg$?
\item
Study \emph{Gro-Tsen positive} games,
where   $G$ is defined to
be Gro-Tsen positive if $GH>0$ whenever $H>0$.
\item
Characterize    $\eg$ as the minimal (or maximal)
equivalence relation $\equiv$ on $\mathbf{PG}$ satisfying a certain set
of conditions. Compare items (c), (d) in Problem \ref{probsi}.
\item
More generally, are there further binary relations satisfying a 
significant set of the properties listed in  Problem \ref{probsi}(c), (d)?
\item
Define recursively $ G \leqa H$ if,
for every Left option $G^L$ of $G$, there is some
Left option $H^L$ of $H$ such that  $ G^L \leqa H^L$ and,
symmetrically, for every Right option $H^R$ of $H$, there is some
Right option $G^R$ of $G$ such that  $ G^R  \leqa H^R$.
Let $G \eqa H$ if both $ G^L \leqa H^L$ and
$ H^L \leqa G^L$. We expect that, in some natural
sense to be still precisely worked out,
$ \eqa $ is ``orthogonal'' to $  \eqi $ 
in $\mathbf{PG}$. 
Provide a precise meaning for the above vague 
intuition, or provide arguments suggesting
that it has little meaning.
\end{enumerate} 
\end{problems}  

\begin{proposition} \labbel{coarsrmk}   
Conway equivalence is the coarsest
relation on $\mathbf{PG}$ which respects addition
and refines the
equivalence relation of having the same outcome (positive,
negative, zero, or fuzzy).
\end{proposition} 

\begin{proof}
Conway equivalence is well-known to have the above properties.
On the other hand, let $\equiv$  be any relation on $\mathbf{PG}$
and suppose that $G \equiv H$.
If $\equiv$ respects $+$, then  
$G + K \equiv H+K$, for every game $K$.
If $ \equiv $-related games have the same outcome, then
$o(G + K) = o( H+K)$.
Thus $G \ceq H$. This shows that 
$\ceq$ is coarser than $\equiv$. 
\end{proof}

In the next proposition we show that a similar remark  holds
for $\eg$.

\begin{proposition} \labbel{egcoar}
  \begin{enumerate}[(i)]    
\item   
  Gro-Tsen relation is the coarsest relation 
on $\mathbf{PG}$ which refines the
 relation of having the same outcome and 
respects both addition
and product. 
\item
Suppose that $G$ and $H$ are games. 
Then $G \eg H$ if and only if 
 $o(GK+L) = o(HK+L)$,
 for all games $K$, $L$.
 \end{enumerate} 
\end{proposition}

  \begin{proof}
(i) By Corollary \ref{finer}(v), $\eg$ respects both $+$
and $\cdot$; moreover, $\eg$ preserves outcomes, since
it is finer that $\ceq$ which itself preserves outcomes. 
 Conversely, suppose that some 
relation $\equiv$ respects both $+$ and $\cdot$
and preserves outcomes. By Proposition  \ref{coarsrmk},
 $\equiv$ is finer than $\ceq$.  
By assumption, if $G \equiv H$, then 
$GK \equiv HK$, hence $GK \ceq HK$,
by the above sentence. Thus 
$G \eg H$ by the definition of $\eg$.
This means that $\eg$ is coarser than $\equiv $.  
 
(ii) If  we keep $K$ arbitrary but fixed,
the condition reads
 $o(GK+L) = o(HK+L)$, for every $L$.
This means exactly
 $GK \ceq HK$, by the definition of 
 $\ceq$. 
\end{proof}

Just as one has $\ceq$ for addition, and $\eg$ for addition and multiplication together,
one could ask about analogous equivalence
relations for other sets of operations, such as
Conway's alternate addition $G:H$ \cite{C}, or Norton's alternate
multiplication \cite{Ke}.

\begin{definition} \labbel{setofop}  
Given a set $S$ of operations on $\mathbf{PG}$, each of which respects $\cong$,
one may define $\equiv_S$ to be the coarsest equivalence relation such that each
operation in $S$ respects $\equiv_S$ and such that $\equiv_S$ refines the
equivalence relation that is having the same comparison to zero (positive,
negative, zero, or fuzzy).
\end{definition}

We always want our equivalence relations to refine the relation of having the
same comparison to zero, because a game's comparison to zero determines which
player has a winning strategy, and it would not make sense to group together
games for which this differs.  With this definition, we obtain $\ceq$ when
$S=\{+\}$, and obtain $\eg$ when $S=\{\cdot,+,-\}$.  If $S=\{:\}$, then in that
case we simply obtain the relation of having the same comparison to zero, as $:$
respects this relation; however, one might obtain a more interesting equivalence
with $S=\{+,:\}$.

So far, we have not yet  proved that $\equiv_S$ always exists. 
Were games not a proper class, one could use a simple
 abstract proof based on taking the transitive closure of 
the union of all the relations that satisfy our
conditions (at least one such relation exists, isomorphism of
games).
Instead, we are going to give an explicit description of  
$\equiv_S$ which does not give rise to set-theoretic
problems.
Recall that, given a set $S$ of operations
(formally, operation symbols), a \emph{term}
appropriate for $S$ 
is  an expression which can be constructed using 
the operations and a set of variables. 

\begin{proposition} \labbel{equiSex} 
Given a set $S$ of operations on $\mathbf{PG}$, 
 $\equiv_S$ is defined by the following condition:
$G \equiv_S H$ if $o(t(G, K_1, \dots, K_n)) = o(t(H, K_1, \dots, K_n))$ 
 for every term $t$ appropriate for $S$  and games 
$K_1, \dots, K_n$. 
 \end{proposition} 

 \begin{proof}  
Arguing in 
a way similar to the proof of Proposition  \ref{coarsrmk}, 
we see that the condition has to be satisfied.
On the other hand, we are going to see that $\equiv_S$, as defined
in the statement of the proposition, respects each 
operation in $S$. Let us deal with, say, a binary
operation $*$.   Assume that $G \equiv_S H$.
To prove $G*K \equiv_S H*K$, we have to show that 
$o(t(G*K, K_1, \dots, K_n)) = o(t(H*K, K_1, \dots, K_n))$
for every term $t$ and games 
$K_1, \dots, K_n$. The expression 
$t(G*K, K_1, \dots, K_n)$ can be written as 
$u(G,K, K_1, \dots, K_n)$, for some other term $u$.   
Since $G \equiv_S H$, we get $o(u(G,K, K_1, \dots, K_n))=
o(u(H,K, K_1, \dots, K_n))$, that is,
 $o(t(G*K, K_1, \dots, K_n)) = o(t(H*K, K_1, \dots, K_n))$,
what we had to show.  Straightforwardly $\equiv_S$
 refines the
 relation of having the same outcome, just consider the
unary term $t(x)=x$.   
\end{proof}

Of course, in  many specific instances, a simple set of terms
might be enough. For example, as in the case of Conway sum,
if $S= \{ * \} $ and $*$ is a binary commutative and associative
operation, then  $G \equiv_S H$ if and only if 
$o(G*K) = o(H *K)$, for every game $K$.  
Another example in which a simple set of terms
works is given by Proposition \ref{egcoar}(ii). 

\begin{problem}
Can one characterize $\equiv_S$ when $S=\{+,:\}$?  When $S=\{+,-,\cdot,:\}$?
When $S=\{:,-,*\}$, where $*$ represents Norton's multiplication, and $-$
represents negation rather than subtraction?  When $S=\{+,:,-,*\}$?  When
$S=\{+,-,\cdot,:,*\}$?
\end{problem}

\begin{remark}
In the cases considered so far, $\equiv_S$ has always been option-regular.  It
seems like this generally ought to be the case for reasonable sets of
operations, for reasons similar to those that $\eg$ is option-regular, but we
will not attempt to formalize this here.
\end{remark}

\begin{problem} \labbel{probaltrpr}  
In order to make the class of combinatorial games a Ring,
one could take another approach: retain the notion of Conway equivalence as
the ``correct'' idea of equality between games, but modify
the definition of Conway product (which does not pass to the Conway quotient).
\arxiv{Recall that the definition of the Conway product is motivated by the fact that 
in an ordered integral Domain (as indeed the surreals turn out to be
with the definition), if $a>a'$ and $b>b'$, then $(a-a')(b-b') >0$,
hence   $a'b + ab' - a'b' < ab$, so that, \emph{for numbers},
it makes sense to take $G^LH + GH^L-G^LH^L$ as a 
Left option of $GH$, and similarly for the other options in 
Conway definition of the product.    

However, the above argument is less motivated for arbitrary games which
are not necessarily numbers.  
For example,
 we might propose to take 
$G^LH + GH^L-G^LH^L$ as a 
Left option of some possible definition $G \odot H$
of a product only in case when

(*) $o(G-G^L) = o(H-H^L)$, where $o$ denotes the outcome of a game.
and symmetrically for the other possible options.
 
For example, in (*) we consider, say, $G^LH + GH^R-G^LH^R$
as a Right option of $G \odot H$ only if
$o(G-G^L) = o(H^R-H)$.   

We do not know whether the above possibility leads
to a well-defined Operation modulo Conway equivalence,
making $\mathbf{PG} / \ceq $ a Ring.
Even were this  the case, we expect such an operation  to be trivial
in many significant situations. }
In any case, Remarks \ref{notint}(a)(b), which apply
in  a very general setting, provide severe limitations
to the success of the project.

\arxiv{As another possible variation, if, in an ordered integral Domain,
$a>a'$, $b>b'$, $a>a'_1$, $b>b'_1$,  $a>a'_2$ and $b<b'_2$,
then $a'b + ab' - a'b' < ab$, $a'_1b + ab'_1 - a'_1b'_1 < ab$
and $-a'_2b - ab'_2 + a'_2b'_2 < - ab$, so that 
$a'b + ab' - a'b'+a'_1b + ab'_1 - a'_1b'_1 -a'_2b - ab'_2 + a'_2b'_2 
< ab+ab-ab= ab$. 
So that Conway argument could suggest to consider, say, also 
$G^LH + GH^L-G^LH^L+
G^{L_1}H + GH^{L_1}-G^{L_1}H^{L_1}
-G^{L_2}H - GH^{R_2}+G^{L_2}H^{R_2}$ as a 
Left option of $GH$ 
(this does not change the definition of the product when restricted
to numbers, but would change the definition for arbitrary games.) }
 
The main problem seems to be that Conway sum has a very
natural game-theoretical interpretation, while no such natural interpretation 
seems to be known for the Conway product. 
One should either find a  game-theoretical interpretation
for the product, or modify the definition, if possible,
in order to get such a game-theoretical interpretation.
\end{problem}

 \section{Two remarks about Sets as Games} \labbel{setsas} 

Just as  games can be considered  ``bilateral'' sets \cite{C,Co},   
 \emph{sets} can be considered as  games in which  Right has no option
in  each subposition, including the improper subposition.
Of course, in any reasonable Theory of sets, there is a definable
Bijection between ``sets''  in the classical set-theoretical sense, and
``sets''  in the above game-theoretical sense;
this is the Correspondence  given by  $S \mapsto \{\,  s \mid \,  \} $,
where $s$ varies among the elements of 
the ``classical'' set  $S$.
(A set could also be considered exactly as an impartial game,
but $\eqi$ has different effects on sets considered in that way.)

\begin{proposition} \labbel{propsets} 
  \begin{enumerate}[(i)]   
 \item    
Suppose that $G$ and $H$ are sets in the above sense of 
Right-subposition-free games. 
Then $G \eqi H$ if and only if $G$ and $H$ are isomorphic 
(=extensionally equivalent). 
\item
In particular, the Semigroup of Sets
with Conway addition can be extended to a Ring $\mathbf{RSets}$,
 the SubRing of     $\mathbf{PG} / {\eqi}$
generated by the Class of Sets\footnote{Formally,
 by $\eqi$-equivalence classes of sets, but we
have just showed in (i) that each equivalence class contains just one 
element.}.
\end{enumerate}
 \end{proposition}  

\begin{proof}
Item (i) is proved by
transfinite induction on the birthday of $G$.
If $G$ and $H$ are sets and $G-H$ 
is iteratively zero, then, for every Left option $G^L$,
Right has an option (necessarily on $-H$, since $G$ is a set) 
such that $G^L +(-H)^R$ is iteratively zero.
But $(-H)^R$ is of the form $-(H^L)$, so by the inductive
assumption
$G^L$ and    $H^L$ are isomorphic. Letting  
$G^L$ vary among all Left options of $G$, 
we get that $G \subseteq H$, as sets.
By the symmetrical  argument, letting Right move first on
$-H$, we get $G \cong H$.
 
(ii) If  $G$ and $H$ are sets, then $G+H$
is a set, since Right has no option in any subposition
of $G+H$. Everything else follows from
(i) and the fact that $\mathbf{PG} / {\eqi}$ is
a Ring, by Theorem \ref{ringth}.
\end{proof}

\begin{remark} \labbel{rmksets}    
(a) In connection with 
Proposition \ref{propsets}(i), 
sets behave in a radically different  way with respect to 
$\ceq$.  
Indeed, any set is Conway equivalent to its
 Von Neumann rank, thus the quotient of the class of
Sets modulo Conway equivalence is isomorphic to the class of
the ordinals.

By the way, the above  example shows that $\eqi$ 
is a much finer relation than $\ceq$.

(b) It is not completely clear which members of 
$\mathbf{PG} / {\eqi}$ are exactly necessary to get
$\mathbf{RSets}$ from Proposition \ref{propsets}(ii). Of course,
a SubClass of $\mathbf{Numb}/ {\eqi}$ from
Proposition  \ref{inumbsur} 
 is enough, since every set is a number
(actually, as we mentioned, Conway equivalent to an ordinal).

(c) Problem: is $\mathbf{RSets}$ an Integral Domain?
If the answer is affirmative, describe its Field of quotients.

(d) In passing, note that various different operations on Sets have been studied;
see the introduction of \cite{Ki} for some references.

(e) As we mentioned, it is straightforward that the
 Class of Sets is closed under Conway sum;
on the other hand, it is not closed under Conway product, not even up
$\eqi$-equivalence (on the other hand, it is closed under product up to Conway 
equivalence). See Proposition \ref{notclo} below.
 
Thus the product of two sets is not necessarily a set;
but in any case is defined in $\mathbf{Numb}/ {\eqi}$,
by Proposition \ref{propsets}(ii). Curiously, a parallel situation occurs
for exponentiation: exponentiation is a well-defined operation
on surreals \cite{Go}, but  ordinals are not closed under
such an operation; actually, there is no  ``natural''  
exponentiation operation on the ordinals satisfying
appropriate algebraic properties \cite{A}. 
\end{remark}

\begin{proposition} \labbel{notclo}   
The Class of Sets is  not closed under Conway product up to
$\eqi$-equivalence.
\end{proposition}  

\begin{proof} 
 Consider the product $2^\circ \cdot 2^\circ =
\{ \,  0, 2^\circ, 2^\circ+2^\circ-1 \mid \, \} $.
First, we check that $2^\circ \cdot 2^\circ$ and
$2^\circ + 2^\circ = \{ \,   2^\circ, 1+2^\circ \mid \, \}$
 are not $\eqi$-equivalent. Otherwise,
$(2^\circ \cdot 2^\circ) - (2^\circ + 2^\circ)$ 
would be iteratively zero. But if Left
makes the move $0$ on the first summand, Right has no possible
reply to make the sum iteratively zero (actually, not even
Conway zero). 

We now can check that
there is no set which is $\eqi$-equivalent to
 $2^\circ \cdot 2^\circ$.
Otherwise, call such a set $S$. 
Then $(2^\circ \cdot 2^\circ) - S$ would be iteratively zero,
thus, if Left as First moves to    $(2^\circ +  2^\circ -1) - S$,
Right would have a move  turning the sum
into an iteratively zero game. 
Right moving on the first component (removing $-1$)
is not apt, since this would imply that 
$(2^\circ +  2^\circ) - S$ is iteratively zero, that is
$2^\circ +  2^\circ \eqi S$. But then $S$ would be simultaneously
$\eqi$-equivalent to $2^\circ \cdot 2^\circ$ and 
to $2^\circ + 2^\circ$. By transitivity, such games
would be $\eqi$-equivalent, but we have seen in the previous
paragraph that this is not the case.

Thus Right needs a move on $-S$, making  
$(2^\circ +  2^\circ -1) + (- S)^R$ iteratively zero.
Changing sign, this means that 
 some Left move  $S^L$ is $\eqi$-equivalent to
$2^\circ +  2^\circ -1$. Note that 
$S^L$ cannot be $2^\circ$, since otherwise,
by Lemma \ref{lemi}(b4), (b5),  
$2^\circ +  2^\circ -1 - 2^\circ 
\eqi 2^\circ -1$ would be iteratively zero,
which is not the case, since $ 2^\circ -1$
is Left winner. 
Now, if Left moves on
$(2^\circ +  2^\circ -1) - S^L$ choosing
$(2^\circ -1) - S^L$, Right cannot reply
with $2^\circ - S^L$, since we have just showed
that $S^L$ cannot be $2^\circ$
(thus $2^\circ - S^L$ is not iteratively zero, since
otherwise $2^\circ \eqi S^L$, but this would mean
that the two games are extensionally equivalent, 
by (a)).
Thus necessarily $(2^\circ -1) - S^{LL}$
is iteratively zero, for some option $S^{LL}$.
Since  $2^\circ -1$ is  Conway equivalent to $1$,
$\eqi$ is finer than $\ceq$ and $1$ is the only set
with Conway value $1$, necessarily   $S^{LL} $
is $1$. But this is impossible,
since Left can move on  $(2^\circ -1) - S^{LL}=
(2^\circ -1) - 1$ to $-1-1$ and Right cannot turn
 $-1-1$ to $0$ in a single move. 
  
We have reached a contradiction assuming that
 $2^\circ \cdot 2^\circ$ is $\eqi$-equivalent to
some set, hence we have proved that
the Class of Sets is  not closed under Conway product, up
to $\eqi$-equivalence.
\end{proof}

\begin{remark} \labbel{mad}
As a quite marginal remark, considering sets as games, 
as we have done in this section, might be perceived 
as  unnatural. In fact,  the two notions
have been developed in very different contexts and with
very different purposes. However, considering sets as
games has sometimes advantages.

Recall that E. Zermelo defined (or interpreted)
natural numbers as the sets $ \emptyset $, $ \{ \emptyset  \} $,
 $ \{ \{ \emptyset  \} \} $, $ \{ \{ \{ \emptyset  \} \}  \} $, 
\dots, while Von Neumann considered
the sets $ \emptyset $, $ \{ \emptyset  \} $,
 $ \{ \,  \emptyset, \{ \emptyset  \}  \, \}  $, 
 $ \{ \,  \emptyset,   \{ \,  \emptyset, \{ \emptyset  \}  \, \} \, \} $, \dots, instead.
On the basis of the above contrasting definitions,
 P. Benacerraf \cite{B}  argued that
natural numbers are not sets, actually, 
that ``there are no such things as numbers''. See also \cite[p. 85]{Ma}.

 Here game theory comes to the rescue!
Indeed, corresponding Zermelo and  Von Neumann's numbers are Conway
equivalent, so that no trouble arises if we consider natural numbers
as games, that is, as Conway equivalence classes of sets.
In passing, note that $2$, as defined in Example \ref{exic}(a) is Zermelo's
number, while   $2^\circ$ is Von Neumann's. Note also that the same
arguments as in Example \ref{exic} show that $2$ and $2^\circ$
are not $\eqi$-equivalent.   
 \end{remark}

\end{document}